\documentclass[preprint,12pt]{elsarticle}

\newcommand{\vertiii}[1]{{\left\vert\kern-0.25ex\left\vert\kern-0.25ex\left\vert #1 
    \right\vert\kern-0.25ex\right\vert\kern-0.25ex\right\vert}}
\usepackage{amsthm}
\usepackage{amsmath}
\usepackage{amssymb}
\usepackage{bm}
\usepackage{mathrsfs}

\usepackage{algorithmic}
\usepackage{algorithm}

\usepackage{color}

\usepackage{url}

\usepackage{supertabular}

\newtheorem{thm}{Theorem}

\newtheorem{lemma}[thm]{Lemma}

\newtheorem{conjecture}[thm]{Conjecture}

\theoremstyle{definition}

\theoremstyle{remark}

\numberwithin{equation}{section}

\numberwithin{thm}{section}

\DeclareMathAlphabet{\mathsfsl}{OT1}{cmss}{m}{sl}

\renewcommand{\phi}{\varphi}

\newcommand{\Expect}{\operatorname{\mathbb{E}}}

\def\Expect{\mathbb{E}}
\def\bA{\mathbf{A}}

\def\bD{\mathbf{D}}
\def\bB{\mathbf{B}}

\def\bC{\mathbf{C}}
\def\bH{\mathbf{H}}

\def\b0{\mathbf{0}}

\def\bI{\mathbf{I}}

\usepackage{algorithmic}
\usepackage{algorithm}

\usepackage{amssymb}

\journal{ }

\begin{document}

\begin{frontmatter}



\title{A note on the matrix arithmetic-geometric mean inequality}


\author{Teng Zhang}

\address{Department of Mathematics, University of Central Florida, Orlando, FL 32654, USA}

\begin{abstract}
This note proves the following inequality: If $n=3k$ for some positive integer $k$, then for any $n$ positive definite matrices $\bA_1,\bA_2,\cdots,\bA_n$, the following inequality holds:
\begin{equation}\label{eq:main}
\frac{1}{n^3}\Big\|\sum_{j_1,j_2,j_3=1}^{n}\bA_{j_1}\bA_{j_2}\bA_{j_3}\Big\|
\geq \frac{(n-3)!}{n!} \Big\|\sum_{\substack{j_1,j_2,j_3=1,\\\text{$j_1$, $j_2$, $j_3$ all distinct}}}^{n}\bA_{j_1}\bA_{j_2}\bA_{j_3}\Big\|,
\end{equation}
where $\|\cdot\|$ represents the operator norm. This inequality is a special case of a recent conjecture proposed by Recht and R\'{e}~\cite{Recht12beneath}.
\end{abstract}


\end{frontmatter}


\section{Introduction}
The classical arithmetic-geometric mean inequality asserts that when $a_1, a_2, \cdots, a_n$ are positive numbers, then  $\frac{1}{n}(\sum_{i=1}^n a_i)\geq (\prod_{i=1}^n a_i)^{\frac{1}{n}}$. This inequality has been generalized from the case of $n$ positive scalars to the case of $n$ positive-definite matrices in various works (see~\cite{Bhatia1990,Bhatia1993,Bhatia2000,Lawson2001,Lawson2011}).

Recently, Recht and R\'{e}~\cite{Recht12beneath} conjectured another generalization of the arithmetic-geometric mean inequality, which is formulated differently from the inequalities proved in previous works. The conjecture is as follows: For $n$ positive definite matrices $\{\bA_i\}_{i=1}^n$, the following inequality holds:
\begin{equation}\label{eq:conjecture}
\frac{1}{n^m}\Big\|\sum_{j_1,\ldots,j_m=1}^{n}\bA_{j_1}\bA_{j_2}\cdots\bA_{j_m}\Big\|
\geq \frac{(n-m)!}{n!} \Big\|\sum_{\substack{j_1,\ldots,j_m=1,\\\text{$j_1, \ldots, j_m$ all distinct}}}^{n}\bA_{j_1}\bA_{j_2}\cdots\bA_{j_m}\Big\|.
\end{equation}
Here $\|\cdot\|$ represents the operator norm, i.e., the largest singular value. Recht and R\'{e} showed that, if this inequality holds, then the randomized coordinate descent algorithm without replacement sampling converges faster than the version with replacement sampling for both least mean squares and randomized Kaczmarz algorithms.

While the conjecture has been proved in \cite[Proposition 3.2]{Recht12beneath} for the special case $n=m=2$, to the best of our knowledge, the conjecture for other cases remains open. 

Inspired by the analysis of randomized algorithm, a very similar conjecture was proposed in \cite{Duchi2012}, which is formulated as
\[
\frac{1}{n^m}\sum_{j_1,\ldots,j_m=1}^{n}\Big\|\bA_{j_1}\bA_{j_2}\cdots\bA_{j_m}\Big\|
\geq \frac{(n-m)!}{n!}\sum_{\substack{j_1,\ldots,j_m=1,\\\text{$j_1, \ldots, j_m$ all distinct}}}^{n}\Big\|\bA_{j_1}\bA_{j_2}\cdots\bA_{j_m}\Big\|.
\]
For this conjecture and the case $m=3$ has been proved recently in \cite{Israel2014}. However, the conjecture is different from \eqref{eq:conjecture} since both LHS and RHS are the sum of operator norms, and the technique used in the proof is not extendable to the proof of \eqref{eq:conjecture}.

The main contribution of this note is a proof of the conjecture~\eqref{eq:conjecture} when $m=3$ and $n=3k$ for any integer $k>1$. The rest of the paper is organized as follows. In Section~\ref{sec:nm3}, we prove the conjecture for $n=m=3$. Section~\ref{sec:generalization} generalizes the result to the case $m=3$ and $n=3k$. Section~\ref{sec:conjecture} discusses other settings such as $m=4$ and some interesting open problems related to the conjecture. 


\section{The proof of the conjecture when $n,m=3$}\label{sec:nm3}

\subsection{Reduction of the conjecture}
In the note, we write  $\bA\geq \bB$ or $\bB\leq \bA$ if $\bA-\bB$ is positive semidefinite. To prove \eqref{eq:main}, WLOG we may assume 
$
\|\sum_{j=1}^{n}\bA_{j}\|=1,
$
which implies
\begin{equation}\label{eq:bound_I}
\sum_{j=1}^{n}\bA_{j} \leq \bI.
\end{equation}

Then the LHS of \eqref{eq:conjecture} is $1/n^m$ and it suffices to prove
\begin{equation}\label{eq:main_alternate}
-\frac{1}{n^m}\bI \leq \Expect [ \bA_{i_1}\bA_{i_2}\cdots\bA_{i_m} ] \leq \frac{1}{n^m}\bI,
\end{equation}
where $\Expect [ \bA_{i_1}\bA_{i_2}\cdots\bA_{i_m} ]$ represents the expected value of $\bA_{i_1}\bA_{i_2}\cdots\bA_{i_m}$ in the probability space that the $n$-tuple $\{i_1,i_2,\cdots,i_n\}$ is sampled randomly from all permutations of $\{1,2,\cdots, n\}$.

\subsection{Proof of the upper bound}\label{sec:upper}
In this section we will prove the second inequality in \eqref{eq:main_alternate} when $n,m=3$, i.e., the upper bound of $\Expect [ \bA_{i_1}\bA_{i_2}\cdots\bA_{i_m} ]$. The proof is based on the following lemmas, and their proofs are deferred to Sections~\ref{sec:lemma1} and~\ref{sec:lemma2}.
\begin{lemma}\label{lemma:lemma1}
For symmetric matrices $\bA$, $\bB\in \mathbb{R}^{p\times p}$, and any positive semidefinite matrix $\bC\in \mathbb{R}^{p\times p}$, \[
\bA\bC\bA+\bB\bC\bB\geq \bA\bC\bB+\bB\bC\bA.\]
\end{lemma}
\begin{lemma}\label{lemma:lemma2}
If $\bA$, $\bB$ and $\bC$ are symmetric matrices of the same size and $\bA\leq \bB$, then \[
\bC\bA\bC\leq \bC\bB\bC.\]
\end{lemma}

To prove the upper bound of $\Expect [ \bA_{i_1}\bA_{i_2}\bA_{i_3} ]$, first we apply Lemma~\ref{lemma:lemma1} and obtain
\begin{equation}\label{eq:lemma1_cor}
\bA_{i_1}\bA_{i_2}\bA_{i_3}+\bA_{i_3}\bA_{i_2}\bA_{i_1}
\leq \bA_{i_1}\bA_{i_2}\bA_{i_1}+\bA_{i_3}\bA_{i_2}\bA_{i_3}.
\end{equation}
Applying \eqref{eq:lemma1_cor}, we have
\begin{align}
&\Expect [ \bA_{i_1}\bA_{i_2}\bA_{i_3} ]
= \frac{1}{2}\,\Expect [ \bA_{i_1}\bA_{i_2}\bA_{i_3} + \bA_{i_3}\bA_{i_2}\bA_{i_1}]
\nonumber\\\leq& \frac{1}{4}\,\Expect [ \bA_{i_1}\bA_{i_2}\bA_{i_3} + \bA_{i_3}\bA_{i_2}\bA_{i_1}+\bA_{i_1}\bA_{i_2}\bA_{i_1}+\bA_{i_3}\bA_{i_2}\bA_{i_3}]
\nonumber\\=&\frac{1}{4}\,\Expect [(\bA_{i_1}+\bA_{i_3})\bA_{i_2}(\bA_{i_1}+\bA_{i_3})],
\label{eq:upper1}\end{align}
where the inequality is obtained by taking the expectation to both LHS and RHS of \eqref{eq:lemma1_cor}. 

On the other hand, since $\bA_{i_1}+\bA_{i_3}\leq \bI-\bA_{i_2}$ (which follows from \eqref{eq:bound_I}),  Lemma~\ref{lemma:lemma2} implies \begin{equation}\label{eq:upper2}
(\bA_{i_1}+\bA_{i_3})\bA_{i_2}(\bA_{i_1}+\bA_{i_3})
\leq (\bA_{i_1}+\bA_{i_3})\big(\bI-(\bA_{i_1}+\bA_{i_3})\big)(\bA_{i_1}+\bA_{i_3}).\end{equation}
Since $\bA_{i_1}+\bA_{i_3}$ and $\bI-(\bA_{i_1}+\bA_{i_3})$ are simultaneously  diagonalizable and their eigenvalues are between $0$ and $1$, applying $\max_{0\leq a\leq 1}a^2(1-a)= 4/27$ we have
\begin{equation}
(\bA_{i_1}+\bA_{i_3})\big(\bI-(\bA_{i_1}+\bA_{i_3})\big)(\bA_{i_1}+\bA_{i_3})
\leq \frac{4}{27}\bI.\label{eq:upper3}
\end{equation}
The upper bound of $\Expect [ \bA_{i_1}\bA_{i_2}\bA_{i_3} ]$ in \eqref{eq:main_alternate} is then proved by combining \eqref{eq:upper1}, \eqref{eq:upper2}, and \eqref{eq:upper3}.
\subsection{Proof of the lower bound}\label{sec:lower}
In this section we will prove the first inequality in \eqref{eq:main_alternate} when $n,m=3$, i.e., the lower bound of $\Expect [ \bA_{i_1}\bA_{i_2}\cdots\bA_{i_m} ]$.

To prove the lower bound of $\Expect [ \bA_{i_1}\bA_{i_2}\bA_{i_3} ]$ in \eqref{eq:main_alternate}, we apply Lemma~\ref{lemma:lemma1} by plugging in $\bA=\bA_{i_1}, \bB=-\bA_{i_3}$, and $\bC=\bA_{i_2}$:
\begin{equation}\label{eq:lower0}
-\bA_{i_1}\bA_{i_2}\bA_{i_3}-\bA_{i_3}\bA_{i_2}\bA_{i_1}
\leq \bA_{i_1}\bA_{i_2}\bA_{i_1}+\bA_{i_3}\bA_{i_2}\bA_{i_3}.
\end{equation}
Similar to \eqref{eq:upper1}, applying \eqref{eq:lower0} we have
\begin{align}
&\Expect [ - \bA_{i_1}\bA_{i_2}\bA_{i_3} ]
= \frac{1}{2}\,\Expect [ - \bA_{i_1}\bA_{i_2}\bA_{i_3} -  \bA_{i_3}\bA_{i_2}\bA_{i_1}]
\nonumber\\\leq& \frac{1}{4}\,\Expect [ - \bA_{i_1}\bA_{i_2}\bA_{i_3} - \bA_{i_3}\bA_{i_2}\bA_{i_1}+\bA_{i_1}\bA_{i_2}\bA_{i_1}+\bA_{i_3}\bA_{i_2}\bA_{i_3}]
\nonumber\\=&\frac{1}{4}\,\Expect [(\bA_{i_1}-\bA_{i_3})\bA_{i_2}(\bA_{i_1}-\bA_{i_3})].
\label{eq:lower1}\end{align}
Therefore, to prove the lower bound in \eqref{eq:main_alternate}, it suffices to show
\begin{equation}
\Expect [(\bA_{i_1}-\bA_{i_3})\bA_{i_2}(\bA_{i_1}-\bA_{i_3})]\leq \frac{4}{27}\bI.
\label{eq:lower2}\end{equation}
To prove \eqref{eq:lower2}, we let $\bD=(\bI-\bA_1-\bA_2-\bA_3)/3$ and $\tilde{\bA}_i=\bA_i+\bD$ for $i=1,2,3$. Then by Lemma~\ref{lemma:lemma2},
\begin{equation}\label{eq:lower3}
\Expect [(\bA_{i_1}-\bA_{i_3})\bA_{i_2}(\bA_{i_1}-\bA_{i_3})]
\leq
\Expect [(\tilde{\bA}_{i_1}-\tilde{\bA}_{i_3})\tilde{\bA}_{i_2}(\tilde{\bA}_{i_1}-\tilde{\bA}_{i_3})].
\end{equation}
Since \[(\tilde{\bA}_{i_1}-\tilde{\bA}_{i_3})\tilde{\bA}_{i_2}(\tilde{\bA}_{i_1}-\tilde{\bA}_{i_3})=2\tilde{\bA}_{i_1}\tilde{\bA}_{i_2}\tilde{\bA}_{i_1}+2\tilde{\bA}_{i_3}\tilde{\bA}_{i_2}\tilde{\bA}_{i_3}-(\tilde{\bA}_{i_1}+\tilde{\bA}_{i_3})\tilde{\bA}_{i_2}(\tilde{\bA}_{i_1}+\tilde{\bA}_{i_3})\]
and
\[
\Expect [2\tilde{\bA}_{i_1}\tilde{\bA}_{i_2}\tilde{\bA}_{i_1}+2\tilde{\bA}_{i_3}\tilde{\bA}_{i_2}\tilde{\bA}_{i_3}]=
\Expect [2\tilde{\bA}_{i_2}(\tilde{\bA}_{i_1}+\tilde{\bA}_{i_3})\tilde{\bA}_{i_2}],\]
we have
\begin{align}
&\Expect [(\tilde{\bA}_{i_1}-\tilde{\bA}_{i_3})\tilde{\bA}_{i_2}(\tilde{\bA}_{i_1}-\tilde{\bA}_{i_3})]
\nonumber\\=&\Expect [2\tilde{\bA}_{i_1}\tilde{\bA}_{i_2}\tilde{\bA}_{i_1}+2\tilde{\bA}_{i_3}\tilde{\bA}_{i_2}\tilde{\bA}_{i_3}-(\tilde{\bA}_{i_1}+\tilde{\bA}_{i_3})\tilde{\bA}_{i_2}(\tilde{\bA}_{i_1}+\tilde{\bA}_{i_3})]
\nonumber\\=&\Expect [2\tilde{\bA}_{i_2}(\tilde{\bA}_{i_1}+\tilde{\bA}_{i_3})\tilde{\bA}_{i_2}-(\tilde{\bA}_{i_1}+\tilde{\bA}_{i_3})\tilde{\bA}_{i_2}(\tilde{\bA}_{i_1}+\tilde{\bA}_{i_3})]\label{eq:lower4}
\end{align}
Applying the fact that $\tilde{\bA}_1+\tilde{\bA}_2+\tilde{\bA}_3=\bI$, we have
\begin{align}
&\Expect [2\tilde{\bA}_{i_2}(\tilde{\bA}_{i_1}+\tilde{\bA}_{i_3})\tilde{\bA}_{i_2}-(\tilde{\bA}_{i_1}+\tilde{\bA}_{i_3})\tilde{\bA}_{i_2}(\tilde{\bA}_{i_1}+\tilde{\bA}_{i_3})]\nonumber\\=&\Expect [2\tilde{\bA}_{i_2}(\bI-\tilde{\bA}_{i_2})\tilde{\bA}_{i_2}-(\bI-\tilde{\bA}_{i_2})\tilde{\bA}_{i_2}(\bI-\tilde{\bA}_{i_2})]
\nonumber\\=&\Expect [-\tilde{\bA}_{i_2}+4\tilde{\bA}_{i_2}^2-3\tilde{\bA}_{i_2}^3] = \frac{4}{27}\bI+
\Expect [-\frac{13}{9}\tilde{\bA}_{i_2}+4\tilde{\bA}_{i_2}^2-3\tilde{\bA}_{i_2}^3],\label{eq:lower5}
\end{align}
where the last step applies $\Expect [\tilde{\bA}_{i_2}]=\frac{1}{3}\bI$. Since $\max_{0\leq x\leq 1} -\frac{13}{9}x +4x^2-3x^3=0$ and the eigenvalues of $\tilde{\bA}_i$ lie in $[0,1]$, we have
\[
-\frac{13}{9}\tilde{\bA}_{i_2}+4\tilde{\bA}_{i_2}^2-3\tilde{\bA}_{i_2}^3\leq \mathbf{0}.
\]
Combining it with \eqref{eq:lower3}, \eqref{eq:lower4} and \eqref{eq:lower5}, \eqref{eq:lower2} is proved and therefore the lower bound in \eqref{eq:main_alternate} is proved.

\subsection{Proof of Lemma~\ref{lemma:lemma1}}\label{sec:lemma1}
The difference of its LHS and RHS can be written as the product of a matrix with its transpose:
\[\bA\bC\bA+\bB\bC\bB - \bA\bC\bB - \bB\bC\bA
= (\bA-\bB)\bC(\bA-\bB)=\Big((\bA-\bB)\bC^{0.5}\Big) \Big((\bA-\bB)\bC^{0.5}\Big)^T,
\]
which is clearly positive semidefinite.
\subsection{Proof of Lemma~\ref{lemma:lemma2}}\label{sec:lemma2}
Since $\bB-\bA$ is positive semidefinite, we can assume $\bB-\bA=\bH\bH^T$ for some matrix $\bH$. Then
$
\bC\bB\bC-\bC\bA\bC=\bC(\bB-\bA)\bC=(\bC\bH)(\bC\bH)^T$ is also positive semidefinite.

\section{Generalization to $n=3k$}\label{sec:generalization}
It is possible to extend the proof from the case $n=m=3$ to the  cases where $m=3$ and $n=3k$ for any positive integer $k$. The proof follows directly from the following lemma.
\begin{lemma}\label{thm:generalization}
If \eqref{eq:main_alternate} holds for $(n,m)=(n_0,m_0)$, then it also holds for $(n,m)=(k\,n_0,m_0)$ with any positive integer $k$.
\end{lemma}
\begin{proof}
If $n=kn_0$ and $m=m_0$, then
\begin{align}
&\Expect [ \bA_{i_1}\bA_{i_2}\cdots\bA_{i_m} ]
 = \frac{1}{k^m} \Expect \Big[\sum_{j=1}^k\bA_{i_j}\cdot\sum_{j=k+1}^{2k}\bA_{i_j}\cdot\sum_{j=2k+1}^{3k}\bA_{i_j}\cdots \sum_{j=(m-1)k+1}^{mk}\bA_{i_j}\Big]\nonumber\\
&= \frac{1}{k^m} \Expect \Big[\Expect_{l_1,l_2,\cdots,l_{n_0}}\Big[\sum_{j=(l_1-1)k+1}^{l_1k}\!\!\!\!\!\!\bA_{i_j}\cdot\!\!\!\!\!\!\sum_{j=(l_2-1)k+1}^{l_2k}\!\!\!\!\!\!\bA_{i_j}\cdot\!\!\!\!\!\!\sum_{j=(l_3-1)k+1}^{l_3k}\!\!\!\!\!\!\bA_{i_j}\cdots \!\!\!\!\!\!\sum_{j=(l_m-1)k+1}^{l_mk}\!\!\!\!\!\!\bA_{i_j}\Big]\Big],\label{eq:generalize1}
\end{align}
where $\{l_1,l_2,\cdots,l_{n_0}\}$ is a random permutation of $\{1,2,\cdots,n_0\}$.

 Apply \eqref{eq:main_alternate} with $(n,m)=(n_0,m_0)$  to $n_0$ positive definite matrices $\{\sum_{j=(l-1)k+1}^{lk}\!\bA_{i_j}\}_{l=1}^{n_0}$, we have
\begin{equation}
-\frac{1}{n_0^m}\bI\leq \Expect_{l_1,l_2,\cdots,l_{n_0}}\Big[\sum_{j=(l_1-1)k+1}^{l_1k}\!\!\!\!\!\!\bA_{i_j}\cdot\!\!\!\!\!\!\sum_{j=(l_2-1)k+1}^{l_2k}\!\!\!\!\!\!\bA_{i_j}\cdot\!\!\!\!\!\!\sum_{j=(l_3-1)k+1}^{l_3k}\!\!\!\!\!\!\bA_{i_j}\cdots \!\!\!\!\!\!\sum_{j=(l_m-1)k+1}^{l_mk}\!\!\!\!\!\!\bA_{i_j}\Big]\leq \frac{1}{n_0^m}\bI\label{eq:generalize2}.
\end{equation}
Combining \eqref{eq:generalize1} and \eqref{eq:generalize2}, we proved \eqref{eq:main_alternate} for $(n,m)=(k\,n_0,m_0)$.
\end{proof}

We remark that since the conjecture for $(n,m)=(2,2)$ has been proved in \cite[Proposition 3.2]{Recht12beneath}, Lemma~\ref{thm:generalization} also implies that the conjecture when $(n,m)=(2k,2)$, i.e., when $n$ is even and $m=2$.

\section{Discussion}\label{sec:conjecture}
Having proved the case $n=m=3$, a natural next step is to prove the case $n=m=4$. The basic idea in this paper is to divide $\sum_{\substack{j_1,\ldots,j_m=1,\\\text{$j_1, \ldots, j_m$ all distinct}}}^{n}\bA_{j_1}\bA_{j_2}\cdots\bA_{j_m}$ into several parts and find the bounds of each parts. For example, in Section~\ref{sec:upper}, we prove the upper bound in \eqref{eq:main_alternate} by finding the upper bounds of $\bA_2\bA_1\bA_3+\bA_3\bA_1\bA_2$, $\bA_1\bA_2\bA_3+\bA_3\bA_2\bA_1$ and $\bA_2\bA_3\bA_1+\bA_1\bA_3\bA_2$ separately. Based on this idea, we have the following conjecture, which would implies \eqref{eq:conjecture} when $n,m=4$.
\begin{conjecture}
For positive semi-definite matrices $\bA$, $\bB$, $\bC$, and $\bD$, the following inequality holds:
\[
\|\bA(\bB\bC+\bC\bB)\bD+\bD(\bB\bC+\bC\bB)\bA\|\leq \frac{1}{64}\|\bA+\bB+\bC+\bD\|^4.
\]
\end{conjecture}
While simulations seem to imply the correctness of this conjecture, the difficulty lies in the construction of an equivalent version of Lemma~\ref{lemma:lemma1} for the product of four matrices.

Another interesting open conjecture is to extend the case $n=m=3$ to other unitarily invariant norms. That is, can we prove that for all positive definite matrices $\bA$, $\bB$ and $\bC$, the following inequality holds:
\[
9 \vertiii{\bA\bB\bC+\bA\bC\bB+\bB\bA\bC+\bB\bC\bA+\bC\bA\bB+\bC\bB\bA}\leq 2\vertiii{\bA+\bB+\bC}^3,
\]
where $\vertiii{.}$ denotes every unitarily invariant norm? While the simulations seem to indicate its correctness, and the approach in this paper only applies to the operator norm.


\section*{References}
  \bibliographystyle{elsarticle-num}
  \bibliography{inequality}





\end{document}